\documentclass[a4paper]{amsart}
\usepackage[utf8x,utf8]{inputenc}
\usepackage{amsfonts,amsmath,amssymb,amsthm}
\usepackage{amscd}
\usepackage{mathrsfs}
\usepackage{ucs,bbm} 
\usepackage{colonequals}
\usepackage[all]{xy}
\usepackage{verbatim}
\usepackage[backref=page]{hyperref}

\hypersetup{
     colorlinks   = true,
     citecolor    = red,
     linkcolor = blue,
     urlcolor = purple
}

\usepackage{enumerate}
\usepackage{xcolor}

\mathchardef\dash="2D

\newcommand{\Q}{\mathbb{Q}}
\newcommand{\QQ}{\mathbb{Q}}

\newcommand{\ZZ}{\mathbb{Z}}

\newcommand{\Mod}{\mathbf{Mod}}
\newcommand{\Map}{\mathbf{Map}}
\newcommand{\GG}{\mathbb{G}}

\DeclareMathOperator{\Gal}{Gal}

\DeclareMathOperator{\Hom}{Hom}

\DeclareMathOperator{\Lie}{Lie}

\DeclareMathOperator{\Ext}{Ext}
\DeclareMathOperator{\Ab}{Ab}

\newcommand{\calM}{\mathcal{M}}

\newcommand{\calO}{\mathcal{O}}

\newcommand{\crys}{\mathrm{crys}}

\newcommand{\dR}{\mathrm{dR}}

\newcommand{\FOg}{\mathbf{FOg}}
\newcommand{\MFOg}{\mathbf{MFOg}}
\newcommand{\MF}{\mathbf{MF}}
\newcommand{\Fil}{\mathrm{Fil}}

\newcommand{\ox}{\otimes}

\newcommand{\et}{{\rm \acute{e}t}}
\newcommand{\ad}{{\rm ad}}

\usepackage{fancyhdr}

\newtheorem{exo}{Exercise}[section]
\newtheorem{thr}[exo]{Theorem}

\newtheorem{prp}[exo]{Proposition}
\newtheorem{crl}[exo]{Corollary}

\theoremstyle{definition}\newtheorem{dfn}[exo]{Definition}
\theoremstyle{remark}\newtheorem{rmk}[exo]{Remark}
\theoremstyle{remark}
\theoremstyle{definition}

\title{The Fontaine--Ogus realisation of Laumon 1-motives}
\author{Nicola Mazzari}
\address{Univ. Bordeaux, CNRS, Bordeaux INP, IMB, UMR 5251,  F-33400, Talence, France}
\email{nicola.mazzari@math.u-bordeaux.fr}
\keywords{Laumon 1-motives, Ogus realisation}
\subjclass[2010]{14F30, 14L15, 14F40}
\date{\today}

\begin{document}
\begin{abstract}
    We construct the (filtered) Ogus realisation of Laumon $1$-motives over a number field. This realisation extends the functor defined on Deligne $1$-motives by Andreatta, Barbieri-Viale and Bertapelle. 
\end{abstract}
\maketitle
\section{Introduction}

Let $K$ be a number field and let $\FOg(K)$ be the $\QQ$-linear abelian category of filtered Ogus structures over $K$ defined in \cite[\S~1.3]{AndBarBer:17a}. The main result of \cite{AndBarBer:17a} is  the existence a realisation functor $T_{\FOg}:\calM_{1,\QQ}(K) \to \FOg(K)$ from the category of Deligne 1-motives over $K$ up to isogeny. Moreover this functor is fully faithful.

In the present article we extend the above realisation $T_\FOg$ to the category $\calM_{1,\QQ}^a(K)$ of Laumon 1-motives over $K$ up to isogeny.

First we prove that $T_\FOg$ factors through a finer category $\MFOg(K)$ obtained by enriching $\FOg(K)$ with the Hodge filtration. We also require the admissibility condition used by Fontaine   defining the category $\MF_K^\ad$ of admissible filtered $\phi$-modules \cite{Fon:94c}. For this reason we  call $\MFOg(K)$ the category of Fontaine--Ogus structures over $K$. More precisely $T_\FOg$ factors through the full subcategory $\MFOg_1(K)\subset \MFOg(K)$ of Fontaine--Ogus modules of level $\le 1$ (see \S~\ref{ssec:MFOg}).

Then we define the category $\MFOg^a_{(1)}(K)$, which is analogous to the category considered in \cite{Bar:07a}, containing $\MFOg_{(1)}(K)$ as a full subcategory. Now we can state the main result of this paper.
\begin{thr}
	There exists a is fully faithful realisation functor
	\[
		T^a:\calM_{1,\QQ}^a(K)\to \MFOg^a_{(1)}(K)
	\]
	extending $T_{\FOg}$.
\end{thr}
	Indeed it is possible to extend $T_\FOg$ to Laumon 1-motives just adding some extra vector spaces to the definition of $\FOg$. The introduction of $\MFOg(K)$ and $\MFOg^a(K)$ is necessary in  order to preserve fully faithfulness.
	
	The definition of $\MFOg^a_{(1)}(K)$ and the full faithfulness depend on   results about the category of Laumon 1-motives in \S~\ref{ssec:devi}: we express the latter  category as an iterated fibre product of categories. There is another important ingredient in the definition of $\MFOg^a_{(1)}(K)$, which is necessary to make the functor full: it depends on the \emph{section} $\gamma_M$ described in Remark~\ref{rmk:gamma}.
	
	 Further, in Corollary of \S~\ref{ssec:devi} we find a shorter proof of the fact that  the cohomological dimension of Laumon 1-motives is $1$ (cf. \cite{Maz:10a}).
\subsubsection*{Acknowledgements} The author is grateful to Alessandra Bertapelle for all the support and the mathematical insights. 
\section{1-motives}
	\subsection{Laumon 1-motives}
	Let $K$ be a (fixed) field of characteristic zero (later it will be a number field).
	Let ${\sf Ab}_K$ be the category of  sheaves of abelian groups on the category of affine $K$-schemes endowed with the fppf topology. We will consider both the category of commutative  $K$-group schemes and that   of formal $K$-group schemes   as  full subcategories of ${\sf Ab}_K$.

		A \emph{Laumon 1-motive} over $K$ (or an effective free 1-motive over $K$, cf. \cite[1.4.1]{BarBer:09a}) is the data of
\begin{enumerate}
	\item A (commutative) formal group $F$ over $K$, such that $\Lie F$ is a finite dimensional $K$-vector space and $F(\bar K)=\lim_{[K':K]<\infty}F(K')$ is a finitely generated and torsion-free $\Gal(\bar K/K)$-module. 
	\item A connected commutative algebraic group scheme ${G}$ over $K$.
	\item A morphism $u:{F}\to {G}$ in the category $\sf Ab_K$.
\end{enumerate}
	\begin{rmk}
		\begin{enumerate}
			\item It is known that any formal $K$-group ${F}$ splits canonically as product ${F}^\circ \times {F}_{\et} $ where ${F}^\circ$ is the identity component of ${F}$ and is a connected formal $K$-group, and ${F}_{\et} = {F} /{F}^\circ$ is \'etale. Moreover, ${F}_{\et}$ admits a maximal sub-group scheme ${F}_{\rm tor}$ , \'etale and finite, such that the quotient ${F}_{\et} /{F}_{\rm tor} = {F}_{\rm fr}$ is étale-locally constant of the type $\ZZ^r$. One says that ${F}$ is torsion-free if ${F}_{\rm tor} = 0$.
			\item By a theorem of Chevalley any connected algebraic group scheme ${G}$ is the extension of an abelian variety ${A}$ by a linear $K$-group scheme ${L}$ that is product of its maximal sub-torus ${T}$ with a vector $K$-group scheme ${V}$. We denote by $G_\times=G/V$ the semi-abelian quotient of $G$. (See \cite{Dem:07a} 
	for more details  on algebraic and formal groups)
		\end{enumerate}
	\end{rmk}
\subsection{Morphisms} We can consider a Laumon 1-motive (over $K$) $M=[u:{F}\to {G}]$ as a complex of sheaves in ${\sf Ab}_K$ concentrated in degree $0,1$. 		A \emph{morphism} of Laumon 1-motives is a commutative square in the category $\Ab_K$. We denote by $\calM^{a}_1=\calM^{a}_{1}(K)$ the category of Laumon $K$-1-motives, i.e. the full sub-category of $C^b(\Ab_K)$ whose objects are Laumon 1-motives.

	 We  define the  category $\calM^{a}_{1,\QQ}(K)={\calM}^{a}_1(K)\ox\Q$ of  Laumon 1-motives up to isogenies by replacing the Hom groups with $\Hom_{{\calM}^{a}_1(K)}(M,M')\ox_\ZZ \QQ $. The category of Laumon 1-motives up to isogenies is  abelian.
	 
	 The category of Deligne 1-motives over $K$ is the full sub-category $\calM_1(K)$ of $\calM_1^{a}(K)$ whose objects are $M=[u:{F}\to {G}]$ such that ${F}^\circ=0$ and ${G}$  is semi-abelian (cf. \cite[\S 10.1.2]{Del:74b}). We can also define the \emph{up to isogeny} version $\calM_{1,\QQ}(K)={\calM}_1(K)\ox\Q$ which is an abelian subcategory of $\calM^{a}_{1,\QQ}(K)$. 
	 \subsection{Devissage of Laumon 1-motives}	\label{ssec:devi} 
	In this section we note $\calM_1^{a}$ for $\calM_1^{a}(K)$. 
	 
	\begin{rmk}
		 Notation as in the previous section. Given a Laumon 1-motive $M=[u:F\to G]$ there is an exact sequence
	\[
		0\to V\to M\to M_\times\to 0
		\]
		where $M_\times=[u_\times:F\to G_\times]$ is obtained by composition with the canonical projection $G\to G_\times =G/V$.   Since $M_\times$ admits a universal vector extension $M_\times^\natural$, the latter  class of the previous extension is determined by a map
		\[
			\alpha_M:U(M)\to V\qquad U(M):=\Ext^1(M_\times,\GG_a)^*\ .
		\]
		
		Let further $M_\et=[u_\et:F_\et\to G_\times]$ be the Deligne 1-motive obtained by restricting $u_\times$ to $F_\et$
	\end{rmk}
We denote by $\Map_K$ the category whose objects are $f:Z_0\to Z_1$ morphisms of finite dimensional $K$-vector spaces. There are two exact functors
\[
	s,t:\Map_K\to \Mod_K,\quad s(f)=Z_0\ , \ t(f)=Z_1 \ .
\]
\begin{prp}\label{prp:??}
	Let $\calM_1^{\times}$ be the full subcategory of Laumon 1-motives whose object are of the form $M=M_\times$.
	
	The association
	\[
		M\to (M_\times, \alpha_M) 
	\]
	induces an equivalence of categories between the category of Laumon 1-motives  and the fibre product category 
	\[
		 \calM_1^\times\times_{\Mod_K}\Map_{K}
	\]
	taken we respect to the following diagram
	\[
		\calM_1^\times\xrightarrow{U} \Mod_K\xleftarrow{s} \Map_K \ .
	\]
	(See \cite[Tag 0030]{Stacks-Project-Authors:2018aa} for details on the fibre product of categories.)
\end{prp}
\begin{proof}
	Faithfulness and essential surjectivity are straightforward. We only need to show that the functor is full. Consider a morphism $\xi_\times:(M_\times,\alpha_M)\to (M_\times',\alpha_{M'})$, i.e. two diagrams
	\[
		\xymatrix{
		F \ar[d]_{u_\times} \ar[r]^f   & F'\ar[d]_{ {u}_{\times}' }  & & U(M)\ar[d]_{\alpha_M}\ar[r]^a & U(M')\ar[d]_{\alpha_{M'}}\\
		G_\times\ar[r]_{g_\times} &{G}_\times' &  &V \ar[r]_v & V'
	}\ .
	\]
	We have to prove the existence of a map $\xi :M\to M'$  inducing $\xi_\times$. If we denote by $M_\times^\natural$ the universal vector extension of $M_\times$ we have a commutative diagram
	\[
		\xymatrix{
		V\ar[d]_v&\ar[l]_{\alpha_M}U(M)\ar[d]_{a}\ar[r]&M_\times^\natural\ar[d]^{\xi_\times^\natural}\\
		V'&\ar[l]_{\alpha_{M'}}U(M')\ar[r]&{M_\times'}^\natural
		}
	\]
	inducing via push-out a morphism $\xi:M\to M'$ with the expected properties. 
\end{proof}
For $u:F\to G$ we denote $du=\Lie(u)$.
\begin{prp}\label{prp:times}
	The association
	\[
		 M_\times=[u_\times:F\to G_\times]\mapsto (M_\et=[u_\et:F_\et\to G_\times],du_\times:\Lie(F)\to \Lie(G_\times))
	\]
	induces an equivalence of categories between the category $\calM_1^{\times}$ and the fibre product category
	\[
		\calM_1\times_{\Mod_K}\Map_K
	\]
	taken with respect to the following diagram
	\[
		\calM_1\xrightarrow{\ell} \Mod_K\xleftarrow{t} \Map_K \ .
	\]
	where $\ell(M)=\Lie(G)$.
\end{prp}
\begin{proof}
	This is an immediate  corollary of \cite[Proposition~1.5.2]{BarBer:09a}.
\end{proof}
\begin{rmk}
	By the previous propositions a  Laumon 1-motive $M$ is equivalent to the data $(M_\et, \alpha_M,du_\times)$.
\end{rmk}
By the previous devissage we easily get the following result (already proved in a more direct way in \cite{Maz:10a})
\begin{crl}
	The category $\calM_{1,\QQ}^{a}$, of Laumon 1-motives up to isogeny (over $K$), is of cohomological dimension  1.
\end{crl}
\begin{proof}
	Given a fibre product of abelian categories $P=X\times_S Y$ (along exact functors) there is a long exact sequence of derived Hom
	\[
		\Ext^i_P\to \Ext^i_X\times \Ext^i_Y\to \Ext^i_S\xrightarrow{+}...
	\]
	We know that $\Mod_K$ (resp. $\Map_K$) has cohomological dimension $0$ (resp. $1$) (see \cite[\S~2.1]{Maz:11a}). It follows, by successively using the previous propositions, that for $i\ge 2$
	\[
		\Ext_{\calM_{1,\QQ}^{a}}^i(M,N)\cong \Ext_{\calM_{1,\QQ}}^i(M_\et,N_\et)
	\]
	and the latter  is zero by \cite[Proposition~3.2.4]{Org:04a}.
	
	Thus the cohomological dimension is at most $1$, and we know that there are non trivial extensions. For instance $\Ext_{\calM_{1,\QQ}^{a}}^1(\ZZ[1],\GG_m)=K^*\otimes_\ZZ \QQ$.
\end{proof}

\section{Adding the Hodge filtration}
\subsection{$p$-adic Hodge theory for 1-motives} It is known \cite[\S~6.3.3]{Fon:94c} that, given a Deligne $1$-motive $M$ over a $p$-adic field $K$, its de Rham realisation $T_\dR(M)$ is naturally endowed with an  admissible filtered $(\phi,N)$-module structure. In the following we are only interested in the case  $K$ is the fraction field of $W(k)$ (for $k$ finite of characteristic $p$) and $M$ has good reduction.  
\begin{prp}
	The de Rham realisation induces a  functor (by abuse of notation we use again $T_\dR$)
	\[
		T_\dR:\calM_1^{\rm good}\to \MF^{\ad}_K \ ,
	\]
	where $\calM_1^{\rm good}$ is the category of Deligne 1-motives, over $K=W(k)[1/p]$, having good reduction.
\end{prp}
\begin{proof}
	We assume that $M$ is the generic fibre of  a \emph{lisse} 1-motive over the dvr $\calO_K$ (i.e. $M$ is of good reduction). We denote by $M_{k}$ its special fibre.  Then by \cite{AndBar:05a} there is a canonical isomorphism
	\[
		T_\dR(M)\cong T_{\crys}(M_{k})\otimes K
	\]
	thus $T_{\dR}(M)$ carries a   Frobenius $\phi$ (see \cite[\S~3.2.1]{AndBarBer:17a}, they note it $F_v$)  and a (1-step) filtration, namely
	\[
	 \Fil^iT_\dR(M)= \begin{cases}
	 	0 &i\ge 0\\
		\ker(T_\dR(M)\to \Lie(G))& i=0\\
		T_\dR(M)& i\le -1
	 \end{cases}  \quad .
	\]
	By devissage w.r.t. to the weight filtration\footnote{admissibility is a property closed under extensions and $T_\dR$ of an exact sequence of $1$-motives gives and exact sequence of filtered vector spaces.} we can easily prove that $T_\dR(M)$ is admissible, since $T_\dR(-)$ of an abelian variety (with good reduction), of a torus (of constant rank over $\calO_K$) and of its Cartier dual, are all admissible (by \cite[Proof of Lemma 3.2.2]{AndBarBer:17a} and \cite{Col:92a}). 
\end{proof}

\subsection{Fontaine--Ogus modules}\label{ssec:MFOg} Let now $K$ be a number field and $M$ be a $1$-motive over $K$. We know that  for some $n>>0$, $M$  can be considered as a lisse 1-motive over $\calO_K[1/n]$ \cite[Lemma~3.1.2]{AndBarBer:17a}. Then, for  all finite and unramified  places $v\nmid n$, $T_{\dR}(M_{K_v})$ can be consider as an object of $\MF^{\ad}_{K_v}$ (by the previous section). 

This motivates the following definition
\begin{dfn}
	Let $\MFOg(K)$ be the category whose objects are systems $(T, \Fil^\bullet)$ such that
	\begin{itemize}
		\item $T=(T_\dR,(T_v)_v)\in \FOg(K)$. We denote by $\phi_v$ the Frobenius on $T_v$
		\item $\Fil^\bullet$ is a (decreasing, exhaustive) filtration on $T_\dR$ (called Hodge filtration).
		\item for almost all $v$, $(T_v,\Fil^\bullet_v=\Fil^\bullet\otimes K_v,\phi_v)$ is an admissible filtered  $\phi$-module over $K_v$.
	\end{itemize}
Morphisms of $\MFOg(K)$ are morphism of $\FOg(K)$ compatible with respect to the ``Hodge'' filtration.
\end{dfn}
\begin{prp}
	The category $\MFOg(K)$ is abelian.
\end{prp}
\begin{proof}
	It is clear how to define kernels and cokernels. We already know that $\FOg$ is abelian, hence morphisms are strict with respect to the weight filtration. We have to prove that morphisms are strictly compatible with respect to the Hodge filtration. This follows form the fact that morphisms are strict in  $\MF_{K_v}^{\ad}$.
\end{proof}
\begin{prp}\label{prp:tmfog}
	The filtered Ogus realisation $T_{\FOg}$ factors through $\MFOg$ via
	\[
		T_{\MFOg}:\calM_{1}(K)\rightarrow \MFOg(K) \ ,
	\]
	induced by 
	\[
		T_{\MFOg}(M)=(T_\FOg(M), \Fil^\bullet T_\dR(M))\ .
	\]
	
	Moreover $T_{\MFOg}$ is fully faithful.
\end{prp}
\begin{proof}
	The de Rham realisation respects the Hodge filtration. To get the full faithfulness we just need to note that the forgetful functor
	\[
		\MFOg(K)\rightarrow \FOg(K)\ , (T,\Fil^\bullet)\mapsto T\ ,
	\]
	is faithful.
\end{proof}
\begin{rmk}
	In \cite{ChiLazMaz:19a} it is proven that the filtered Ogus realisation $T_{\FOg}$ extends to the category of Voevodsky motives. Also the latter functor $T_{\MFOg}$ can be extended to Voevodsky motives. In fact it is straightforward to add the Hodge filtration and it is possible to prove by devissage the required admissibility condition. 
\end{rmk}
\section{Extending the realisation to Laumon 1-motives}

Let us denote simply by $T:\calM_1\to \MFOg$ the realisation functor defined in the previous section. We aim to extend this functor to the category $\calM_1^a$ of Laumon 1-motives. For this reason we have to introduce another category $\MFOg^a_{(1)}$ containing $\MFOg_{(1)}$ as a full subcategory and such that there exists a functor $T^a:\calM_1^a\to \MFOg^a_{(1)}$ extending $T$.
\subsection{The target category} Recall that $\FOg_{(1)}$ is the category of filtered Ogus structure of level $\le 1$ \cite[Definition~1.4.4]{AndBarBer:17a}. Then we can define $\MFOg_{(1)}$ to be the subcategory of $\MFOg$ given by $(T,\Fil^\bullet)$ such that $T$ is of level $\le 1$, $\Fil^{1}=0$ and $\Fil^{-1}=T$.
\begin{dfn}
	Let $\MFOg_{(1)}^a$ be the category of systems $(T,\Fil^\bullet, \alpha:A_0\to A_1, \beta:B_0\to B_1,\delta,\gamma)$ where
	\begin{itemize}
		\item $(T,\Fil^\bullet)$ is in $\MFOg_{(1)}$.
		\item $\alpha,\beta$ are objects of $\Map_K$.
		\item   $\delta:B_1\cong T/\Fil^0$ is an isomorphism.
		\item $\Fil^0\subset A_0$ and $\gamma:A_0\to T$ is a $K$-linear map.
		\item The following diagram is cartesian
			\[
			\xymatrix{
		 T\ar[r]& T/\Fil^0\\
	  A_0\ar[u]^\gamma \ar[r] & B_0\ar[u]^{\delta\beta}
		}
			  \]
	\end{itemize}
	
	By abuse of notation we simply write $(T,F^\bullet, \alpha, \beta)$ to denote such an object.
	
		Morphisms are compatible systems of maps.
\end{dfn} 
\begin{rmk}
	  Note that $\MFOg_{(1)}$ is a full subcategory of $\MFOg_{(1)}^a$ via
	\[
		(T,\Fil^\bullet)\mapsto (T,\Fil^\bullet,0,0)\  ,
	\] 
	where the first zero map is $\Fil^0\to 0$, while the second is $0\to T/\Fil^0$.
\end{rmk}
\begin{rmk}\label{rmk:gamma}
Given a Laumon 1-motive $M$ we have the following splitting
\[
	\xymatrix{
	0\ar[r]& \ell(M_\et) \ar[r]& \ell(M_\times^\natural)\ar[r]^{\rm pr}& \ar@{.>}@/^1pc/[l]^{du_\times^\natural}\Lie(F)\ar[r]&0\ .
	}
\]
Hence $x\mapsto x-(du_\times^\natural\circ {\rm pr})(x)$ gives a map $\ell(M_\times^\natural)\to \ell(M_\et)$. We denote by $\gamma_M$ its restriction to $U(M)\subset\ell(M_\times^\natural)$.

Thus we can consider the  object of $\MFOg_{(1)}^a$, naturally associated to $M$, represented by the following diagram
	\[
	\xymatrix{
0\ar[r]& U(M_\et)\ar@{=}[d] \ar[r]& T_\dR(M_\et)\ar[r]^\pi& \Lie(G_\times)\ar[r]&0\\
0\ar[r]	  & U(M_\et) \ar[r] & U(M) \ar[d]^{\alpha_M}\ar[u]^\gamma\ar[r] & \Lie(F)\ar[u]^{du_\times}\ar[r]&0\\
& & V &
}
	  \]
	  (See \cite[\S~3.2]{BarBer:09a})
	  \end{rmk}
\begin{thr}
	Let $M=[u:F\to G]$ be a Laumon $1$-motive over $K$. Let
	\[
		T^a(M)=(T_{\MFOg}(M_\et),\alpha_M, du_\times,\gamma_M)
	\]
	be the object of $\MFOg_{(1)}^a$ represented by the above diagram. Then $T^a$ induces a  fully faithful functor
	\[
		T^a:\calM_{1,\QQ}^a\to \MFOg^a_{(1)}
	\]
	extending $T$.
\end{thr}
\begin{proof}
	The non obvious part is the fullness of the functor. This follows from the fullness of $T_{\MFOg}$ and the equivalences of categories in \S~\ref{ssec:devi}. More precisely, let $M,M'$ be two Laumon 1-motives. Then a morphism $\psi\in \Hom_{\MFOg_1^a}(T^a(M),T^a(M'))$ is given by
	\begin{itemize}
		\item a morphism $\eta:T(M_\et)\to T(M_\et')$ in $\MFOg$;
		\item two $K$-linear maps $a:U(M)\to U(M')$, $v:V\to V'$ such that $v\circ\alpha_M=\alpha_{M'}\circ a$;
		\item two $K$-linear maps $b:\Lie F\to \Lie F'$, $c:\Lie(G_\times)\to \Lie (G_\times ')$ such that $c \circ d u_\times=d u_\times'\circ  b$
	\end{itemize}
	satisfying the obvious compatibility conditions. 
	
	By Proposition~\ref{prp:tmfog}  there exists $\xi_\et=(f_\et,g_\times):M_\et\to M_\et'$  (morphism of Deligne 1-motives up to isogeny) such that $\eta=T(\xi_\et)$. Note that by compatibility $dg_\times=c$ and by Proposition~\ref{prp:times} the data of $\xi, b,c$ uniquely determines a morphism $\xi_\times:M_\times\to M_\times'$. To conclude we use Proposition~\ref{prp:??} since $a$ is completely determined by $b$ and $\eta$.
\end{proof}

\begin{rmk}(Comparison with sharp de Rham)
	Consider the following functor
	\[
		S:\MFOg_{(1)}^a\rightarrow \Mod_K\ ,\  (T,\Fil, \alpha,\beta)\mapsto (T\times^{\Fil^0} A_0) \times A_1
	\]
	where $T\times^{\Fil^0} A_0$ is the push-out. Then it is easy to check that $T_\sharp(M)= S( T^a(M))$, where $T_\sharp:\calM_{1,\QQ}^a\to \Mod_K$ is the sharp de Rham realisation \cite[\S~3.2]{BarBer:09a}.
\end{rmk}

\end{document}